\documentclass[12pt]{amsart}
\usepackage{amsmath}
\usepackage{amssymb}
\usepackage{epsfig}
\usepackage{graphicx}
\numberwithin{equation}{section}

\input xy
\xyoption{all}

\calclayout
\allowdisplaybreaks[3]

\theoremstyle{plain}
\newtheorem{prop}{Proposition}

\newtheorem{theo}[prop]{Theorem}
\newtheorem{coro}[prop]{Corollary}

\newtheorem{lemm}[prop]{Lemma}

\theoremstyle{definition}

\def\lra{\longrightarrow}
\def\ra{\rightarrow}

\def\cI{{\mathcal I}}

\def\cK{{\mathcal K}}
\def\cL{{\mathcal L}}
\def\cM{{\mathcal M}}

\def\cO{{\mathcal O}}

\def\cQ{{\mathcal Q}}

\def\cV{{\mathcal V}}
\def\cW{{\mathcal W}}
\def\cX{{\mathcal X}}
\def\cY{{\mathcal Y}}

\def\bP{{\mathbb P}}

\def\bZ{{\mathbb Z}}

\def\bN{{\mathbb N}}
\def\bC{{\mathbb C}}

\def\bF{{\mathbb F}}

\def\CH{\mathrm{CH}}
\def\Bl{\mathrm{Bl}}
\def\Pic{\mathrm{Pic}}

\def\Gr{\mathrm{Gr}}

\def\lim{\mathrm{lim}}

\def\Prym{\mathrm{Prym}}

\def\IJ{\mathrm{IJ}}

\def\wcV{\widetilde{\cV}}
\def\wcX{\widetilde{\cX}}

\makeatother
\makeatletter

\author{Brendan Hassett}
\address{Department of Mathematics\\
Brown University \\
Box 1917 
151 Thayer Street
Providence, RI 02912 \\
USA}
\email{bhassett@math.brown.edu}

\author{Yuri Tschinkel}
\address{Courant Institute\\
                New York University \\
                New York, NY 10012 \\
                USA }
\email{tschinkel@cims.nyu.edu}

\address{Simons Foundation\\
160 Fifth Avenue\\
New York, NY 10010\\
USA}

\title[On stable rationality]{On stable rationality of Fano threefolds and del Pezzo fibrations}

\begin{document}
\date{January 25, 2016}

\maketitle

\section{Introduction}

Recent breakthroughs of Voisin \cite{Voisin}, developed and amplified by
Colliot-Th\'el\`ene--Pirutka \cite{ct-pirutka,ct-pir-cyclic},
Beauville \cite{beau-6}, and Totaro \cite{totaro-JAMS}, have reshaped
the classical study of rationality questions for higher-dimensional
varieties. Failure of stable rationality is now known for large 
classes of rationally-connected threefolds. The key tool is (Chow-theoretic)
integral decompositions of the diagonal, which necessarily exist
for stably rational varieties. Integral decompositions of the diagonal
specialize well, even to mildly singular varieties, connecting 
logically the stable rationality of various classes of varieties.
This puts a premium on discovering appropriate degenerations linking
different classes of rationally connected varieties.
Using these techniques, we prove:

\begin{theo} \label{theo:main}
Let $X$ be a very general smooth non-rational Fano threefold over $\bC$.
Assume that $X$ is not birational to a cubic threefold. Then $X$ is not stably rational. 
\end{theo}

While smooth cubic threefolds are all known to be non-rational, determining whether or not they are
stably rational remains an open problem. No smooth cubic threefolds are known to be stably rational.
However, Voisin \cite{VoisinJEMS} has shown that the cubic threefolds where her techniques
fail to apply, i.e., those  admitting an integral 
decomposition of the diagonal, are dense in moduli. 

Several common geometric threads, developed in collaboration with
Kresch, unify our
approach to Theorem~\ref{theo:main}. In \cite{HKT-conic},
we showed that very general conic bundles over rational surfaces 
with sufficiently large discriminant fail to be stably rational.
The conic bundle structures on cubic threefolds 
arising from projection from a line have quintic plane curves as their discriminants---too small for our techniques to apply.
Nevertheless, conic bundles are a useful tool for analyzing stable
rationality of Fano threefolds.
Second, in \cite{HKT} we classified quartic del Pezzo surfaces with
mild singular fibers and maximal monodromy; previously \cite{HTCEJM}
we showed that a number of these arise as specializations of
Fano threefolds of index one. Together, these facilitate a
streamlined approach to most families of non-rational Fano threefolds.

{\bf Acknowledgments:} The first author was supported by 
NSF grant 1551514. We are grateful to Andrew Kresch for his
foundational contributions that made this research possible.
We also benefitted from conversations with Alena Pirutka.
Kresch and Pirutka also offered helpful feedback on early drafts of
this paper.

\section{Organization of the cases}
\label{sect:org}

The tables in \cite{isk-prokhorov} enumerate
non-rational Fano threefolds; see also the summary in \cite[Section 2.3]{beau-survey}, which includes references to methods used to establish non-rationality. 
In the results that follow, `very general' refers to the complement to a 
countable union of Zariski-closed proper subsets of the families enumerated
in this section.

Let $V$ be a Fano threefold of Picard rank one. 
Write $\Pic(V)=\bZ h$, for some ample $h$, and express the anti-canonical
class $-K_V=rh$. Let $h^{1,2}=\dim H^1(V,\Omega^2_V)$, which
equals the dimension of the intermediate Jacobian $\IJ(V)$.
We enumerate non-rational Fano threefolds $V$ of Picard rank one, using the invariants
$(r, -K_V^3, h^{1,2})$.

\begin{itemize}
\item $(1,2,52)$: double cover of $\bP^3$ ramified in a surface of degree 6,
unirationality is unknown, very general $V$ are not stably rational \cite{beau-6}
\item $(1,4,30)$: quartic in $\bP^4$, unirationality is unknown, very general $V$ are not stably rational \cite{ct-pirutka} 
\item $(1,6,20)$: intersection of a quadric and a cubic, unirational
\item $(1,8,14)$: intersection of three quadrics in $\bP^6$, unirational
\item $(1,10,10)$: section of $\Gr(2,5)$ by a subspace of codimension 2 and a quadric, general such $V$ are non-rational, all are unirational
\item $(1,14,5)$: section of  $\Gr(2,5)$ by a subspace of codimension 5, unirational
\item $(2, 8, 21)$: $V_1$, unirationality is unknown
\item
$(2, 8\cdot 2, 10)$:
$V_2$,  double cover of $\bP^3$, ramified in a smooth quartic, unirational, 
very general $V_2$ are not stably rational \cite{Voisin,beau-survey}
\item $(2, 8\cdot 3, 5)$: $V_3$, cubic in $\bP^4$, unirational
\end{itemize}

We next list non-rational minimal Fano threefolds of Picard rank $\ge 2$,
using the invariants $(-K_V^3, h^{1,2})$. All of these admit conic bundle
structures, induced by projection onto a rational surface.

\begin{itemize}
\item $(6,20)$: double cover of $\bP^1\times \bP^2$ branched in a divisor of bidegree $(2,4)$, unirational
\item $(12,9)$: divisor in $\bP^2\times \bP^2$ of bidegree $(2,2)$,  or a double cover of $\bF(1,2)\subset \bP^2\times \bP^2$ of bidegree $(1,1)$, ramified in 
$B\in |-K_{\bF(1,2)}|$, unirational
\item $(14,9)$: double cover of $V_7\simeq \Bl_p(\bP^3)$, branched in $B\in |-K_{V_7}|$, unirational 
\item $(12,8)$: double cover of $\bP^1\times \bP^1\times \bP^1$
branched in a divisor of degree  $(2,2,2)$, unirational
\end{itemize}

\section{Conic bundles over rational surfaces}
\label{sect:HKT}
 
We recall the set-up for the results of \cite{HKT-conic}:
Let $S$ be a smooth projective rational surface over $\bC$. 
Fix a linear system $\cL$ of effective divisors on $S$ such
that the generic member is smooth and irreducible. Consider 
the space of pairs
$$
\left\{\begin{array}{cc}
D\in \cL \text{ nodal and reduced, } 
D'\ra D
 \text{ \'etale of degree two}
\end{array} \right\} \ra \cL
$$
and let $\cM$ be one of its irreducible components. 
Assume it contains a
point $\{D'\ra D\}$ with the following
properties: 
\begin{itemize}
\item{the nodes of $D$ are disjoint from the base locus of $\cL$;}
\item{$D$ is reducible and for each irreducible component $D_1 \subset D$
the induced cover $D'\times_D D_1 \ra D_1$ is non-trivial.}
\end{itemize}

Results of Artin and Mumford \cite{ArtMum} and Sarkisov \cite{Sar}
allow us to assign to each point of $\cM$ a 
conic bundle $X\ra S$, unique up to birational equivalences
over $S$. Essentially, $\cM$ parametrizes ramification
data for the associated Brauer elements in the function
field of $S$, which determine them as $S$ is rational.
The condition on the distinguished point implies that the corresponding
conic bundle has non-trivial Brauer group.

Using Voisin's decomposition of the diagonal technique,
we proved in \cite{HKT-conic} that a very general point $[X] \in \cM$
parametrizes a threefold that fails to be stably rational.

We first observe an obvious strengthening of the main theorem of
\cite{HKT-conic}: $\cM$ need not dominate the linear series $\cL$ but
can be any smooth irreducible parameter space of reduced
nodal curves $D\in \cL$ with \'etale double covers $D'\ra D$.
Let $\cK$ denote the image of $\cM$ in $\cL$, so we have
$$\cM \stackrel{\varphi}{\ra} \cK \subset \cL$$
where $\varphi$ is \'etale and a covering space over
the open subset parametrizing smooth curves.
We still insist that there is a reducible member
whose nodes are disjoint from the base locus of $\cK$,
such that the cover over each component is non-trivial.

Second, our result is easiest to apply in cases where the monodromy
action is large, e.g., when $\cM$ parametrizes {\em all}
non-trivial double covers of the generic point $[D] \in \cK$, 
or equivalently, when the monodromy representation on 
$H^1(D,\bZ/2\bZ) \setminus \{0\}$
is transitive. This is the case when $S=\bP^2$ and $\cL$
parametrizes plane curves of even degree;
in odd degree there are two such orbits \cite{Beau86}.
Large monodromy actions make it easier to decide
which component contains
a given distinguished point $\{D'\ra D\}$.

\section{Classification of quartic del Pezzo fibrations and stable rationality}
\label{sect:quarticdP}

Consider quartic del Pezzo surface fibrations
$\pi:\cX \ra \bP^1$ satisfying two non-degeneracy conditions:
\begin{itemize}
\item{the discriminant is square-free, i.e., $\cX$ is regular and 
the degenerate fibers are complete intersections of two quadrics
with at most one ordinary singularity;}
\item{the monodromy action on the Picard groups of the fibers
is the full Weyl group $W({\mathsf D}_5)$.}
\end{itemize}
The fundamental invariant of such fibrations is the {\em height}
$$h(\cX)=\deg(c_1(\omega_{\pi})^3)=-2\deg(\pi_*\omega_{\pi}^{-1}),$$
an even integer (see \cite{HKT,HTCEJM} for more background). 
The principal results we require are \cite[Th.~10.2]{HKT}:
\begin{itemize}
\item{under the non-degeneracy conditions we have $h(\cX)\ge 8$;}
\item{when $h(\cX)=8$ or $10$, the moduli space of these fibrations
has two irreducible components;}
\item{when $h(\cX)\ge 12$ the moduli space is irreducible.}
\end{itemize}

When $h(\cX)=8,10,12$ the total space $\cX$ is either 
rational or birational to a cubic threefold; see
\cite[\S11]{HKT} and \cite[\S8-10]{HTCEJM} for details.
Thus we will focus on fibrations with heights at least fourteen.
Note that Alexeev \cite{alexeev} established non-rationality
in these cases by relating the del Pezzo fibrations
to conic bundles. (We will review this below.)

\begin{theo} \label{theo:dPnotSR}
Let $\cX \ra \bP^1$ be a fibration in
quartic del Pezzo surfaces satisfying our
non-degeneracy conditions, with $h(\cX)\ge 14$,
and very general in moduli.
Then $\cX$ fails to admit an integral decomposition
of the diagonal and thus is not stably rational.
\end{theo}
\begin{proof}
We first reduce to the conic bundle case, following Alexeev.
Choose a section $\sigma:\bP^1 \ra \cX$, which we may assume
is not contained in any line of the generic fiber. 
Blowing up this section gives a cubic surface fibration with a
distinguished line and projecting from this line gives a conic
fibration:
$$\begin{array}{ccccc}
\cL & \hookrightarrow & \wcX & \stackrel{\pi}{\ra} & S \\
    & \searrow & \downarrow & \swarrow &  \\
    &	       & \bP^1 &   &
\end{array}$$
Here $S\ra \bP^1$ is a rational ruled surface.

The conic bundle structure over $S$ yields a discriminant
curve $D\subset S$ and an \'etale double cover $D'\ra D$.
Note that $D'\ra D$ coincides with the spectral data introduced
in \cite[\S\S2,8]{HKT} and $S$ is the
natural ruled surface containing $D$ described
in \cite[\S10]{HKT}.

Using \cite[\S6]{HKT}\, we pin down the numerical invariants:
Suppose first that $h(\cX)=4n+2$ for $n\ge 3$. Here the surface
$S\simeq \bF_1$, the Hirzebruch surface. Let $\xi$ denote
the $(-1)$-curve and $f$ the class of a fiber.
Then $[D]=5\xi+(n+3)f$ which has genus $h(\cX)-4$.
If $h(\cX)=4n$ for $n\ge 4$ then 
$S\simeq \bF_0\simeq \bP^1 \times \bP^1$. Here $D$ has 
bidegree $(n,5)$, also of genus $h(\cX)-4$.

The fundamental dictionary between del Pezzo fibrations
and spectral data \cite[Th.~10.1]{HKT} implies that the
$D \subset S$ arising from del Pezzo fibrations
are generic in the linear series $\cL=|D|$. The
analysis of \cite[\S3]{HKT}
shows that the monodromy acts on $H^1(D,\bZ/2\bZ)$ via
the full symplectic group, hence transitively on the non-trivial
elements. To apply the main result of Section~\ref{sect:HKT},
it suffices to exhibit a distinguished point in $|D|$,
i.e., a reducible curve $D=D_1\cup D_2$ with $D_1$ and $D_2$
smooth of positive genus, intersecting transversally.

Note that any \'etale double cover on $D_1\cup D_2$
deforms to the parameter space $\cM$.
Indeed, such double covers correspond to 
homomorphisms
$$H_1(D_1\cup D_2,\bZ) \ra \bZ/2\bZ;$$
a specialization $D\rightsquigarrow D_1\cup D_2$
induces a homomorphism 
$$H_1(D,\bZ)\ra H_1(D_1\cup D_2,\bZ)$$
collapsing vanishing cycles, thus a double cover of $D$ 
via composition.

Producing the reducible curve is simple:
For $S=\bF_1$ take $D_1 \in |2\xi+3f|$, the proper
transform of a cubic plane curve, and $D_2 \in |3\xi+nf|$,
a smooth curve of genus $2n-5 \ge 1$.
For $S=\bF_0\simeq \bP^1 \times \bP^1$ take
$D_1$ of bidegree $(2,2)$, an elliptic curve, and $D_2$
of bidegree $(n-2,3)$, of genus $2n-6\ge 2$. 
\end{proof}

\section{Quartic del Pezzo surfaces of height $22$}
 
A quartic del Pezzo surface of height $22$ may be constructed as follows
\cite[\S 4, Case 5]{HTCEJM}:
Let $V=\cO_{\bP^1} \oplus \cO_{\bP^1}(1)^{\oplus 4}$ and consider
the injection
$$V^* \hookrightarrow \cO_{\bP^1}^{\oplus 9}$$
associated with the global sections of $V$. Then we have morphisms
$$\bP(V^*) \hookrightarrow \bP^1 \times \bP^8 \stackrel{\pi_2}{\lra} \bP^8,$$
where the composition collapses the distinguished section
$$
\sigma:\bP^1 \ra \bP(V^*)
$$
arising from the $\cO_{\bP^1}$ summand.
We use $\pi=\pi_1$ for the fibration over $\bP^1$.
Let $\xi=c_1(\cO_{\bP(V^*)}(1))$ and $h=\pi^*(c_1(\cO_{\bP^1}(1)))$
so that $\xi^5=4\xi^4h$.

A generic height $22$ quartic del Pezzo $\cX \ra \bP^1$ admits
an embedding
$$\begin{array}{ccc}
\cX & \hookrightarrow & \bP(V^*) \\
    & \searrow    & \downarrow \\
    &		  & \bP^1 
\end{array}
$$
as a complete intersection of divisors of degrees $2\xi-h$ and
$2\xi$. Let $\cQ \ra \bP^1$ denote the former divisor, which is
canonically determined. It necessarily contains the section 
$\sigma$. The second divisor $\cQ'$ is a pull-back of a quadric hypersurface
via $\pi_2$; it is typically disjoint from $\sigma$.

Consider projection from the section $\sigma$:
$$\varpi: \bP(V^*) \ra \bP(\cO_{\bP^1}(-1)^{\oplus 4}) \simeq \bP^1 \times \bP^3$$
inducing a birational map 
$$\cQ \stackrel{\sim}{\dashrightarrow}\bP^1 \times \bP^3.$$
Restricting to $\cX$ yields a generically finite morphism
$$\phi:\cX \ra \bP^3.$$
We compute its invariants via intersections in $\bP(V^*)$. 
The pullback of the hyperplane class on $\bP^3$ via $\phi$
is $\xi-h$.
First, we have
$$\deg(\phi)=(\xi-h)^3(2\xi)(2\xi-h)=2$$
which means $\phi$ is a double cover. Its ramification divisor
$$R=K_{\cX}-\phi^*K_{\bP^4}=-\xi+h+4(\xi-h)=3(\xi-h)$$
maps to the branch surface $B\subset \bP^3$ of degree six. 

We interpret when $\phi$ fails to be finite. Points $p\in \bP^3$ 
correspond to line subbundles
$$\sigma(\bP^1) \hookrightarrow \cL(p)\simeq 
\bP(\cO_{\bP^1} \oplus \cO_{\bP^1}(-1))\simeq \bF_1 \hookrightarrow \bP(V^*)$$
where $\bF_1$ is the blowup of the projective plane at a point.
Thus $\cQ \cap \cL(p)$ is the union of the $(-1)$-curve and the
proper transform of a line $\ell$ and $\cQ'\cap \cL(p)$ is the 
proper transform of a conic {\em disjoint} from the $(-1)$-curve.
These typically meet at two points but the conic might contain
the line $\ell$, i.e., $\phi^{-1}(p)=\ell$; 
this is a codimension-three condition on $p$ and corresponds
to singular points of $B$. 

How many singularities do we expect on $B$? For the moment, assume
these are ordinary double points for generic $\cX$. We have
\cite[Cor.~3.5]{HTCEJM}:
$$h^1(\Omega^2_{\cX})=22-5=17$$
but for a generic sextic double solid $\cV$ we have 
$h^1(\Omega^2_{\cV})=52$ \cite{isk-prokhorov}. If $\cV$ admits
$n$ ordinary singularities
and rank $r$ class group then we have
$$
52=n-r+1+h^1(\Omega^2_{\wcV}),
$$
where $\wcV\ra \cV$ is the blowup of the singularities.
This can be seen by comparing the topological Euler characteristics and
class groups of $\cV$ and $\wcV$.
We must have $r=2$ if $\cV$ is a contraction of a del Pezzo
fibration $\cX\ra \bP^1$. Thus we expect $n=36$.

\begin{prop}
Let $y_0,y_1,y_2,y_3$ denote coordinates on $\bP^3$.
\begin{itemize}
\item{
The equation for $B$ takes the form $\det(M)=0$
where 
$$M= \left( \begin{matrix} L^2  & Q_0 & Q_1 \\
			   Q_0 & Q'_{00} & Q'_{01} \\
			   Q_0 & Q'_{01} & Q'_{11}
		\end{matrix}
    \right),$$
with $L$ linear in the $y_i$ and the remaining entries
quadratic.}
\item{
The generic such matrix arises from a height $22$ fibration
in quartic del Pezzo surfaces.}
\item{
The singularities of $B$ are of two types. The first type corresponds
to the vanishing of the $2\times 2$ minors of $M$; there
are $32$ such singularities. The second type corresponds to 
the locus 
$$L=Q_0=Q_1=0;$$
there are four such singularities.
}
\end{itemize}
\end{prop}
\begin{proof}
Let $x_0$ and $x_1$ denote homogeneous coordinates on $\bP^1$
and their pullbacks to $\bP(V^*)$. Designate generating
global sections
$$y_0,y_1,y_2,y_3 \in \Gamma(\cO_{\bP(V^*)}(\xi-h))\simeq 
\Gamma(\cO_{\bP^3}(1))$$
and 
$$z,x_0y_0,x_1y_0,\ldots,x_0y_3,x_1y_3 \in \Gamma(\cO_{\bP(V^*)}(\xi)).$$
After completing the square to eliminate the term
linear in $z$, the defining equation $\cQ'$ may be written in the
form
$$
z^2=Q'_{00}x_0^2+2Q'_{01}x_0x_1+Q'_{11}x_1^2,
$$
where the $Q'_{ij}$ are quadratic in the $y_i$.
The defining equation for $\cQ$ takes the form
$$
-zL(y_0,y_1,y_2,y_3)+Q_0x_0+Q_1x_1=0,
$$
where $L$ is linear and $Q_0$ and $Q_1$ are quadratic in the $y_i$.
Eliminating $z$ we obtain
$$x_0^2(Q'_{00}L^2-Q_0^2)+2x_0x_1(Q'_{01}L^2-Q_0Q_1)+
x_1^2(Q'_{11}L^2-Q_1^2)=0,$$
which is the defining equation for the image of $\cX$ in
$\bP^1 \times \bP^3$. The discriminant of this polynomial--regarded
as a binary quadratic form in $x_0$ and $x_1$--can be written as
$$L^4((Q'_{01})^2-Q'_{00}Q'_{11})+L^2(-2Q'_{01}Q_0Q_1+Q'_{00}Q_1^2
+Q'_{11}Q_0^2).$$
After dividing out by $-L^2$ we obtain $\det(M)$.
This proves the first assertion. Reversing the algebra gives
the second assertion.

We analyze the singularities of the hypersurface $\det(M)=0$. 
In general, the singularities of the determinant of
a symmetric $3\times 3$ 
matrix of forms is given by the vanishing
of the $2\times 2$ minors. In geometric terms, this
is the Veronese surface $\operatorname{Ver} \hookrightarrow \bP^5$ 
which has degree four. If the entries are quadratic forms in 
$y_0,\ldots,y_3$
then the image of the associated morphism $\bP^3 \rightarrow \bP^5$
has degree eight. Bezout's Theorem gives $32$ points of 
intersection. 

However, we also have to take into account singularities of the 
entries. Given the form of the upper-left entry
of $M$, these occur precisely when $L=0$. (The other
entries are generic.) The determinantal hypersurface thus
has additional singularities along the locus $L=Q_0=Q_1=0$.
Our genericity assumption implies this is a complete intersection,
thus we obtain four additional ordinary double points.
\end{proof}
 
We have the following corollary:
\begin{coro} \label{coro:sexticheight22}
The generic sextic double solid arises as a deformation of a 
nodal birational
model of a generic height $22$ fibration $\cX \ra \bP^1$ in 
quartic del Pezzo surfaces.
\end{coro}

\section{Index one Fano threefolds}
Let $V$ be a smooth Fano threefold with $\Pic(V)=\bZ K_V$, i.e.,
with rank one and index one. Its degree $d(V)=-K_V^3$
takes the following values \cite{isk-prokhorov}:
$$
d(V)=2, 4, 6, 8, 10, 12, 14, 16, 18, 22.$$
For each $d(V)$ there is an irreducible parameter space for
the corresponding Fano threefolds.
The cases $d(V)=12, 16, 18,$ and $22$ are rational. 

When $d(V)=14$ the generic $X\subset \bP^9$ arises as a linear section 
of the Grassmannian $\Gr(2,6)$. Projective duality gives
a codimension ten section of the Pfaffian cubic hypersurface 
in $\bP^{14}$,
a cubic threefold $V'$. There is a birational map
$V \dashrightarrow V'$; see \cite[\S1]{IlMark}, for example, for
additional details. 
This example is related to quartic del Pezzo fibrations: 
One of the two species of quartic del Pezzo fibrations of height
ten $\cX \ra \bP^1$ admits a natural morphism \cite[\S 11]{HKT}
$$\cX \ra V \subset \bP^9;$$
the image is a nodal Fano threefold of degree $14$.
However, stable rationality of cubic threefolds (and birationally
equivalent varieties) remains an open problem.

\subsection{$d(V)=2$: Sextic double solids}

Failure of stable rationality in this case has been established 
by Beauville \cite{beau-6} and by Colliot-Th\'el\`ene--Pirutka
\cite{ct-pir-cyclic}. It also follows naturally
from our formalism:
Apply Corollary~\ref{coro:sexticheight22} to reduce to the corresponding
del Pezzo fibration of height $22$.
This realizes a generic height $22$ fibration in quartic del
Pezzo surfaces as a nodal sextic double solid. If a smooth threefold
admits an integral decomposition of the diagonal the same holds
true for a specialization with nodes \cite[Th.~1.1]{Voisin}.
Since a very general del Pezzo fibration of height $22$
lacks such a decomposition (by Theorem~\ref{theo:dPnotSR}),
the same holds for a very general sextic double solid.

\subsection{$d(V)=4$: Quartic threefolds}

Failure of stable rationality in this case has been established
by Colliot-Th\'el\`ene and Pirutka \cite{ct-pirutka}. 
To see this from the perspective of fibrations in 
quartic del Pezzo surfaces, it suffices to recall
that a generic such fibration of height $20$ admits a 
birational model as a determinantal quartic threefold with
sixteen nodes \cite[\S11 ]{HTCEJM} (cf.~\cite[Th.~11]{Cheltsov-NNQT}). 
This lacks an 
integral decomposition of the diagonal (Theorem~\ref{theo:dPnotSR}),
so very general quartic threefolds have the same property,
hence fail to be stably rational. 

\subsection{$d(V)=6$: Complete intersections of a quadric
and a cubic in $\bP^5$}
We proceed as before, using the fact that a generic quartic
del Pezzo fibration of height $18$ admits a birational model
as a complete intersection $\cY \subset \bP^5$ with eight nodes.

Indeed, realize 
$$\cX \subset \bP(\cO_{\bP^1}^{\oplus 4} \oplus \cO_{\bP^1}(-1))
\subset \bP^1 \times \bP^5$$
as a complete intersection of forms of bidegree $(1,1)$,
$(0,2)$, and $(1,2)$, as in Case 3 of \cite[\S4]{HTCEJM}.
(Here we are using the irreducibility of the moduli space
of quartic del Pezzo fibrations of height $18$.)
Let $\cY \subset \bP^5$ denote the image of projection
onto the second factor. 
Consider first the image in $\bP^5_{[x_0,\ldots,x_5]}$ of the locus cut out
by the forms of bidegree $(1,1)$ and
$(1,2)$:
$$s L_0+tL_1=sQ_0+tQ_1=0,
$$
with 
$$L_0,L_1 \in \bC[x_0,\ldots,x_5]_1, \quad 
Q_0,Q_1 \in \bC[x_0,\ldots,x_5]_2.
$$
Its equation is obtained by eliminating $s$ and $t$, which yields
$$L_1Q_0-L_0Q_1=\det \left(\begin{matrix} L_0 & L_1 \\ Q_0 & Q_1 \end{matrix}
\right) =0.$$
This is a cubic fourfold $\cW$ singular along the elliptic quartic curve
$$C=\{L_0=L_1=Q_0=Q_1=0\}.$$
Let $Q \in \bC[x_0,\ldots,x_5]_2$ be the form of bidegree $(0,2)$,
so that
$$\cY=\cW \cap \{Q=0\}.$$
This is a complete intersection of $Q$ with $\cW$, having eight nodes
at the intersection $C\cap \{Q=0\}=\{p_1,\ldots,p_8\}.$
The preimages of these nodes in $\cX$ are distinguished sections
of $\cX \ra \bP^1$.

We establish failure of stable rationality for very general
complete intersections as in the previous cases: Use Theorem~\ref{theo:dPnotSR}
to deduce the failure of integral decomposition of the diagonal
for very general quartic del Pezzo fibrations of height eighteen.
It follows that very general complete intersections $V\subset \bP^5$
of a quadric and a cubic also lack such decompositions, so stable
rationality fails.

\subsection{$d(V)=8$: Complete intersections of three quadrics in $\bP^6$}
Let $V\subset \bP^3$ denote a complete intersection of three
quadrics. Beauville \cite[\S6.4,\S6.23]{Beau77} has shown
that $V$ is birational
to a conic fibration
$$X\ra \bP^2,$$
with discriminant $D\subset \bP^2$ of degree seven, and a generic plane
curve of degree seven arises in this way. Thus the results in \S\ref{sect:HKT}
apply: a very general such $V$ fails to be stably rational.

For unity of exposition, we offer a second proof 
using the fact that a generic quartic
del Pezzo fibration of height $16$ admits a birational model
as a complete intersection $\cY \subset \bP^6$ with four nodes.

Express  
$$\cX \subset \bP(\cO_{\bP^1}^{\oplus 3} \oplus \cO_{\bP^1}(-1)^{\oplus 2})
\subset \bP^1 \times \bP^6$$
as a complete intersection of two forms of bidegree $(1,1)$,
and two quadratic forms from $\bP^6_{[x_0,\ldots,x_6]}$ (see 
Case 3 of \cite[\S 4]{HTCEJM}).
Again, $\cY$ is the projection of $\cX$ onto the second factor.
Write the forms of bidegree $(1,1)$ as
$$sL_0+tL_1=sM_0+tM_1=0, \quad L_0,L_1,M_0,M_1 \in \bC[x_0,\ldots,x_6]_1.$$
Eliminating $s$ and $t$ gives a quadratic equation
$$L_1M_0-L_0M_1=\det \left(\begin{matrix} L_0 & L_1 \\ M_0 & M_1 \end{matrix}
	\right)=0.$$
Let $\cW \subset \bP^6$ denote the resulting quadric hypersurface;
it is singular along the plane
$$P=\{L_0=L_1=M_0=M_1=0\} \subset \bP^6.$$
Then $\cY$ arises as the intersection of $\cW$ with the zeros of
two arbitrary $Q_0,Q_1 \in \bC[x_0,\ldots,x_6]_2$. 
It is a complete intersection of three quadrics with singular
locus
$$P \cap \{Q_0=Q_1=0\}=\{p_1,\ldots,p_4\}.$$
The preimage of the singularities are distinguished sections 
of $\cX \ra \bP^1$.  

The argument for failure of stable rational for a very general
complete intersection of three quadrics proceeds as before.

Alternatively, we may observe that such an intersection is birational 
to a conic bundle over $\bP^2$, ramified in a curve $D$ of degree 7, 
and apply results in $\S \ref{sect:HKT}$ directly.

\subsection{$d(V)=10$: Complete intersections in $\Gr(2,5)$}
Fano threefolds $V$ of this type are 
obtained by intersecting the Grassmannian $\Gr(2,5) \subset \bP^9$
with two linear forms and one quadratic form.

In \cite[\S11]{HTCEJM} we showed that generic quartic del Pezzo
fibrations of height fourteen are birational to $\cY \subset \bP^7$,
where $\cY$ is a specialization of $V$ with two nodes.
Repeating the arguments above, we conclude that the very general
such $V$ fails to be stably rational.

\section{Fano threefolds of index two}
In this section we consider Fano threefolds $V$ with
$\Pic(V)=\bZ \frac{K_V}{2}$, i.e., those of
rank one and index two. Here the degree
$d(V)=-K_V^3=8\cdot \delta(V)$ where $\delta(V)\in \bN$.
The possible values are
$\delta(V)=1,2,3,4,5$; if $\delta(V)=4$ or $5$ then $V$
is rational, and when $\delta(V)=3$ then $V$ is a cubic threefold.

\subsection{$\delta(V)=1$: Double cover of Veronese cone}
\label{subsect:DCVC}
Let 
$$\bP:=\bP(1,1,1,2) \subset \bP^6$$ 
denote the cone over 
the Veronese surface $\operatorname{Ver}\subset \bP^5$;
the vertex $p=[0,0,0,1]$ is a terminal singularity of $\bP$ of index $2$. 
Let $B \subset \bP$  denote the restriction of a generic cubic hypersurface in $\bP^6$,
which has degree six in the natural grading on $\bP$. Consider the double cover
$$\phi: V \ra \bP$$
branched over $B$. It is also ramified over $p$; its preimage $v_0 \in V$
is smooth.  We may regard $V$ as a hypersurface in $\bP(1,1,1,2,3)$ of
degree six, which is clearly Fano of index two.
Note that $h^1(\Omega^2_V)=21$ \cite[\S 12]{isk-prokhorov}
and that $V$ depends on $34$ parameters.

We specialize $B$ so it contains $v_0$, analyzing the resulting double cover
$\phi:V\ra \bP$. (This imposes one condition so the construction depends
on $33$ parameters.)
Blowing up $p$ gives a resolution
$$\beta:\bP(\cO_{\bP^2} \oplus \cO_{\bP^2}(-2))\simeq \Bl_p(\bP) \ra \bP.$$
Let $\xi$ and $h$ generate the Picard group of the projective bundle,
where $h$ is the pullback of the hyperplane class on $\bP^2$ and 
$\xi=c_1(\cO_{\bP(\cO\oplus \cO(-2))}).$ Let 
$$E\simeq \bP^2 \subset \bP(\cO_{\bP^2}\oplus \cO_{\bP^2}(-2))$$
denote the exceptional divisor; note that $[E]=\xi-2h$. Let $\tilde{B}$
denote the proper transform of $B$ with $[\tilde{B}]=2\xi+2h$,
$\tilde{V}\ra \bP(\cO_{\bP^1} \oplus \cO_{\bP^1}(-2))$
the double cover branched along $\tilde{B}$, and
$\tilde{E}\subset \tilde{V}$ the preimage of $E$. 
Note that $\tilde{E}\simeq \bP^1\times \bP^1$ as $\tilde{B}$
meets $E\simeq \bP^2$ in a plane conic $C$. Moreover, applying
the Hurwitz formula and adjunction yields
$$N_{\tilde{E}/\tilde{V}} \simeq \cO_{\bP^1 \times \bP^1}(-2,-2).$$
The induced birational morphism $\tilde{V} \ra V$ resolves $v_0$
with exceptional divisor $\tilde{E}\simeq \bP^1 \times \bP^1$.
In particular, $\tilde{V} \ra V$ is universally $\CH_0$-trivial
(see \cite[Prop.~1.8]{ct-pirutka}).
For the equivalence between universal $\CH_0$-triviality and the existence
of integral decompositions of the diagonal, see \cite[Lemma~1.3]{ACTP} and 
\cite[\S1]{Voisin}.

We compute the invariants of $\tilde{V}$: The bundle structure
$$\varpi: \bP(\cO_{\bP^2} \oplus \cO_{\bP^2}(-2)) \ra \bP^2$$
induces a morphism
$$\psi:\tilde{V} \ra \bP^2.$$
Since $\tilde{B}$ is a bisection of $\varpi$, $\psi$ endows
$V$ with the structure of a conic bundle. Let $D\subset \bP^2$ denote
its discriminant curve, which coincides with the branch locus
of $\varpi:\tilde{B} \ra \bP^2$. An adjunction computation
implies $K_{\tilde{B}}=h|_{\tilde{B}}$ so $D$ is a plane octic curve,
generically smooth. 

Now $D$ and $C$ are tangent at each point
of their intersection, i.e., 
$$D\cap C = 2(z_1+\ldots+z_8)=2Z$$
with $\cI_Z\subset \cO_D$ the ideal sheaf.
Thus $D$ depends on $44-8-3=33$ parameters; moreover, the parameter
space of smooth
octic plane curves eight-tangent to $C$ is
birational to a projective bundle over $C^{[8]}$, thus irreducible.
Furthermore, $\eta:=\cO_D(1)\otimes \cI_Z$ 
is a two-torsion element of the Jacobian of $D$.
The double cover
$D' \ra D$
associated with $\tilde{V} \ra \bP^2$ is classified
by $\eta$. From it, we read off the cohomology
of $\tilde{V}$:
$$\IJ(\tilde{V})=\Prym(D' \ra D).$$
The curve $D$ has genus $21$ so $h^1(\Omega^2_{\tilde{V}})=20$.
Thus the singularity $v_0$ reduces this Hodge number by one.

\begin{lemm}
There exists a specialization
$$D \rightsquigarrow D_1 \cup D_2$$
of octic curves eight-tangent to $C$, such that $D_1$ and $D_2$
are transverse plane quartics each four-tangent to $C$. 
This satisfies the requirements of \S\ref{sect:HKT}.
\end{lemm}
\begin{proof}
Consider the space of pairs $(D_1,D_2)$ where $D_1$ and $D_2$ are
plane quartics four-tangent to $C$, with $D_1$ and $D_2$ 
meeting transversally. Repeating the argument above, the plane quartics
four-tangent to $C$ are birational to a $\bP^6$ bundle over $C^{[4]}$,
an irreducible rational variety of dimension ten. It is easy to check 
that a generic pair of such curves meets transversally, yielding a
rational parameter space of dimension twenty. Write
$$D_1\cap C= 2(z_1+z_2+z_3+z_4) \quad
D_2\cap C= 2(z_5+z_6+z_7+z_8)$$
and $\cI_Z\subset \cO_{D_1\cup D_2}$ for the ideal sheaf of 
$Z=\{z_1,\ldots,z_8\}$. Thus $\eta_0:=\cI_Z(1)$ is two-torsion
in the Picard group of $D_1\cup D_2$ and restricts to non-trivial
two-torsion elements on $D_1$ and $D_2$ because
$$z_1+z_2+z_3+z_4\not \equiv [\cO_{D_1}(1)], \quad
z_5+z_6+z_7+z_8\not \equiv [\cO_{D_2}(1)].$$
Otherwise, these four-tuples of points would be collinear.

Linear algebra shows that we can smooth $D_1\cup D_2$ to a 
smooth plane octic $D$ tangent to $C$ at $z_1,\ldots,z_8$.
As we saw in the proof of Theorem~\ref{theo:dPnotSR}, this gives rise to
a cover $D'\ra D$ classified by the divisor $\eta$.
As $D\rightsquigarrow D_1\cup D_2$ we have $\eta \rightsquigarrow \eta_0$.
\end{proof}

Thus the results of \S\ref{sect:HKT} imply that $\tilde{V}$ 
fails to admit an integral decomposition of the diagonal. 
An application of the results of \cite[\S1]{ct-pirutka} implies
that a very general $V\subset \bP(1,1,1,2,3)$ also
fails to admit an integral decomposition of the diagonal, and thus
is not stably rational.

\subsection{$\delta(V)=2$: Quartic double solids}
\label{sect:QDS}
Let $V$ be a quartic double solid
$$\phi:V \ra \bP^3$$
with branch locus a degree four K3 surface $B$.
When $V$ is smooth we have $h^1(\Omega^2_V)=10$.
Voisin \cite{Voisin} and Colliot-Th\'el\`ene--Pirutka \cite{ct-pir-cyclic}
established the failure of stable rationality for very general
varieties in this class.
Here we discuss how to approach this through conic bundle fibrations.

Now suppose $V$ (or equivalently $B$) has a node $p$
and write $\tilde{V}=\Bl_p(V)$. Projection from $p$
gives a conic bundle structure
$$\pi:\tilde{V} \ra \bP^2$$
branched along a sextic plane curve $D$.
The plane curve is typically smooth but admits a six-tangent
conic curve $C$ corresponding to the exceptional divisor of the
induced resolution of $B$. Write
$$D\cap C = 2Z, \quad Z=z_1+\cdots+z_6 $$
so that $\eta:=\cO_D(1)\otimes \cI_Z$ is two-torsion on $D$.
Here $\cI_Z$ is the ideal sheaf of $Z$.

As we saw in the previous case, the parameter space of sextic
plane curves six-tangent to a prescribed conic is irreducible,
being a projective bundle over the Hilbert scheme $C^{[6]}$.
We can specialize 
$$
D \rightsquigarrow D_1 \cup D_2,
$$
where $D_1$ and $D_2$ are smooth plane cubics meeting transversally,
each three-tangent to $C$. Thus $\eta$ specializes to a two-torsion
divisor on $D_1\cup D_2$ that is non-trivial on each component.

An application of the results of \S\ref{sect:HKT} implies that
the very general quartic double solid fails to have an
integral decomposition of the diagonal, thus is not stably rational.

\section{Fano threefolds of higher Picard rank}
As before we write $d(V)=-K_V^3$.

\subsection{$d(V)=6,h^{1,2}(V)=20$}
The first case is double covers 
$$
V\ra \bP^1\times \bP^2
$$ 
branched 
over a divisor of bidegree $(2,4)$. These depend on 
$$
3\times 15-(1+3+8)=33
$$
parameters. Projection onto the second factor
gives a conic bundle structure $V\ra \bP^2$ with octic discriminant
curve $D\subset \bP^2$. The equation of $D$ is given by the vanishing
of the determinant of a $2\times 2$ {\em symmetric} matrix of 
quartic forms in three variables. In particular, $D$ is not general
in its linear series and each symmetric determinantal octic
comes with a distinguished non-trivial two-torsion class, i.e.,
the one associated with the ramification data of $V\ra \bP^2$.
This makes it hard to apply the methods of \S\ref{sect:HKT} directly.

However, there is a natural degeneration of such Fano threefolds
to another class of rationally connected varieties:
Fix distinct points $p,q \in \bP^2$ and consider 
divisors $B_0 \subset \bP^1 \times \bP^2$ of bidegree $(2,4)$
whose fibers over $\bP^1$ admit nodes at $p$ and $q$. (Equivalently,
these are singular along $\bP^1 \times \{p,q\}$.) Consider the 
birational map 
$$\bP^2 \dashrightarrow \bP^1 \times \bP^1$$
blowing up $p$ and $q$ and blowing down the line joining them. This
takes quartic plane curves singular at $p$ and $q$ to bidegree 
$(2,2)$ curves in $\bP^1 \times \bP^1$. Using the induced birational
map
$$\bP^1 \times \bP^2 \dashrightarrow \bP^1 \times \bP^1 \times \bP^1,$$
we see that $B_0$ is mapped to $(2,2,2)$ divisor in the image. Conversely,
$(2,2,2)$ divisors in the image all arise from this construction.

\begin{lemm}
Let $V_0$ denote the 
double cover of $\bP^1\times \bP^2$ branched over a very general
such $B_0$. Let $\widetilde{V_0} \ra V_0$ denote the blowup
along $\bP^1\times \{p,q\}$. If $\widetilde{V_0}$
admits no integral decomposition of the diagonal
then the same holds for the very general Fano variety
$V$ arising as a double cover of $\bP^1 \times \bP^2$
branched over a divisor of bidegree $(2,4)$.
\end{lemm}
\begin{proof} The singularities of
$V_0$ are along the lines $\ell_p:=\bP^1 \times \{p\}$ and 
$\ell_q:=\bP^1 \times \{q\}$.
The singularities of $B_0$
are ordinary double points along the general points of these lines
and cusps (analytically isomorphic to $x^2+y^3=0$) at a finite
number of special points.
For special $r\in \bP^1$ the local singularity type of $V_0$
at $(r,p)$ is of the form
$$w^2=x^2+ty^2+y^3,$$
where $\{x,y\}$ are local coordinates of $\bP^2$ centered at $p$,
$t$ is a local coordinate of $\bP^1$ centered at $r$,
and $w$ is used to realize the double cover over $\bP^1 \times \bP^2$.
Thus the singularities of $V_0$ are resolved by blowing up the lines
$$\widetilde{V_0}=\Bl_{\ell_p \cup \ell_q}(V_0) \ra V_0.$$
The exceptional fibers over the generic points of $\ell_p$
and $\ell_q$
are smooth conics; the fibers over special points are reducible conics.
This computation is similar to, but simpler than,
the singularity analysis of \cite[App.]{ct-pirutka}.

The key is the exceptional fibers are universally
$\CH_0$-trivial, in the sense of \cite[D\'ef.~1.2]{ct-pirutka}.
Applying the result on universal $\CH_0$-triviality
in \cite[\S1]{ct-pirutka}, we conclude that
$V$ fails to be universally $\CH_0$-trivial
if the same
holds for $\widetilde{V_0}$.
(See \cite[Lemma~1.3]{ACTP} and 
\cite[\S1]{Voisin} for the equivalence with integral decompositions
of the diagonal.)
\end{proof}

In \S\ref{subsect:222} we will show that very general double covers
of $\bP^1 \times \bP^1 \times \bP^1$ branched over a $(2,2,2)$
divisor do not admit integral decompositions of the diagonal.

\subsection{$d(V)=12,h^{1,2}(V)=9$}
The first part of the
second case is realized as a divisor in $\bP^2 \times \bP^2$ of
bidegree $(2,2)$, depending on $19$ parameters. Using either projection,
we obtain a conic bundle over $\bP^2$ with sextic discriminant. 
It is well known that the plane sextic can be chosen generically
\cite[\S9]{vG,HaVA}. The main result of \S\ref{sect:HKT} implies that very
general conic bundles over $\bP^2$ with sextic discriminant fail to be 
stably rational. That is, for a very general pair $(D,D'\ra D)$,
where $D$ is a plane sextic and $D'\ra D$ is a non-trivial \'etale
double cover, the corresponding conic bundle $X\ra \bP^2$ fails
to be stably rational. It follows that for very general $D$,
{\em every} double cover $D'\ra D$ is associated with a conic
bundle that is not stably rational. In particular,
this applies to the very general divisor of bidegree $(2,2)$
in $\bP^2 \times \bP^2$.

The second part of the second case is
a double cover $V$ of a hypersurface 
$\bF(1,2) \subset \bP^2\times \bP^2$
of bidegree $(1,1)$ 
branched over an anticanonical divisor $B$ of
bidegree $(2,2)$. This depends on 
$18$ parameters. Again, either projection induces a conic bundle
$V\ra \bP^2$ with sextic discriminant curve $D$. The Mori-Mukai classification
\cite{MoriMukai-manuscripta} shows this is a specialization
of the first part. 

The surface $B$ is a lattice polarized K3 surface of type
$$\Phi:=\begin{array}{c|cc}
 	& f_1 & f_2 \\
\hline
f_1  & 2 & 4 \\
f_2  & 4 & 2 
\end{array}
$$
and the generic such surface arises as a complete intersection
of forms of bidegree $(1,1)$ and $(2,2)$ in $\bP^2 \times \bP^2$. 

The branch curve of $\pi_i|B:B\ra \bP^2$
coincides with the locus where $\pi_i|V:V\ra \bP^2$ fails
to be smooth, i.e., the discriminant curve $D$.
This is a sextic plane curve that is not of general moduli---the
associated K3 double cover has Picard rank two. The technique
of \S\ref{sect:HKT} does not immediately apply in this case.

It is easy to use degeneration techniques to reduce this to cases
where there is no integral decomposition of
the diagonal. Consider a quartic surface $B_0\subset \bP^3$
with nodes $n_1$ and $n_2$ and minimal resolution
$\widetilde{B_0}$. Its Picard group takes the form
$$
\begin{array}{c|ccc}
	 & h & R_1 & R_2 \\
\hline
h & 4 & 0 & 0 \\
R_1 & 0 & -2 & 0 \\
R_2 & 0 & 0 & -2
\end{array}
$$
where $h$ is the pullback of the hyperplane class and the $R_i$
are the exceptional divisors.  The lattice $\Phi$ embeds into
this lattice
$$f_1=h-R_1, \quad f_2=h-R_2.$$
Projection from the nodes $n_1$ and $n_2$ gives a morphism
$$\widetilde{B_0} \ra \bP^2 \times \bP^2$$
with image $B_{\circ}$ a complete intersection of hypersurfaces of degrees 
$(1,1)$ and $(2,2)$. This extends to a birational map
$$\bP^3 \dashrightarrow 
\bF(1,2):=\bP(\Omega^1_{\bP^2}(1))\subset \bP^2 \times \bP^2$$
onto the divisor of bidegree $(1,1)$. To summarize:
\begin{lemm}A double cover of $\bP^3$ branched over a generic 
quartic surface with two nodes is birational to a double
cover of the complete flag variety $\bF(1,2)$ along a special anticanonical
divisor. 
\end{lemm}
By \cite[Th.~1.1]{Voisin} a very general such quartic double solid
fails to admit an integral decomposition of the diagonal.
Thus the same holds for very general Fano threefolds 
$V\ra \bF(1,2)$ and so these fail to be stably rational.

\subsection{$d(V)=14,h^{1,2}(V)=9$ }
The third case is the double cover of $\bP^3$ blown up at a point, with
anticanonical branch locus $B$ meeting the exceptional divisor transversally.
These conic bundles were addressed in \S\ref{sect:QDS} as degenerate
quartic double solids.

\subsection{$d(V)=12,h^{1,2}(V)=8$} \label{subsect:222}
The fourth case is a double cover 
$$
V\ra \bP^1 \times \bP^1 \times \bP^1
$$ 
branched over a divisor of degree $(2,2,2)$. These depend on 
$27-1-9=17$ parameters. For each projection onto
$\bP^1 \times \bP^1$ we obtain a conic bundle, with discriminant $D$ of
bidegree $(4,4)$.
Note that this is not generic; it has
equation
$$D=\{\det(M)=0 \}, \quad M=\left( \begin{matrix} M_{11} & M_{12} \\
						  M_{12} & M_{22} \end{matrix} 
				\right),$$
with 
$$M_{11},M_{12},M_{22} \in
				\Gamma(\cO_{\bP^1\times \bP^1}(2,2)).$$
This may be interpreted geometrically:
the K3 double cover 
$$
B\ra \bP^1 \times \bP^1
$$ 
has Picard group
$$\Pi:=\begin{array}{c|ccc}
	& E_1 & E_2 & E_3 \\
\hline
E_1 & 0 & 2 & 2 \\
E_2 & 2 & 0 & 2 \\
E_3 & 2 & 2 & 0
\end{array}
$$
In this basis, $D \equiv 4(E_1+E_2)$.
The curves $E_3$ and $2(E_1+E_2)-E_3$ are conjugate under the
involution associated with the first two factors,
which fixes $D$. Thus $E_1+E_2-E_3$ restricts to a two-torsion
divisor $\eta$ on $D$, which classifies the double cover $D'\ra D$.

Given that $D$ is not generic in its linear series, the techniques
of \S\ref{sect:HKT} do not apply directly. Clearly the monodromy
cannot act transitively on the non-trivial two-torsion of $D$,
as there is a distinguished element $\eta$. Keeping track of what
happens to $\eta$ as $D\rightsquigarrow D_1\cup D_2$ can be
delicate.

The quickest proof that the very 
general $V\ra \bP^1 \times \bP^1\times \bP^1$ fails to admit a
decomposition of the diagonal is via degeneration of
the branch locus. Let $B_0\subset \bP^3$
denote a quartic surface with nodes $n_1,n_2,n_3$ and minimal 
resolution $\widetilde{B_0}$. Let $h$ denote the pullback of the polarization
and $R_1,R_2,R_3$ the exceptional divisors over $n_1,n_2,n_3$:
$$\begin{array}{c|cccc}
	& h  & R_1 & R_2 & R_3 \\
\hline
h & 4 & 0 & 0 & 0 \\
R_1 & 0 & -2 & 0 & 0 \\
R_2 & 0 & 0 &  -2 & 0 \\
R_3 & 0 & 0 & 0 &  -2  
\end{array}$$
Note that $R_0=h-R_1-R_2-R_3$ is also a smooth rational curve in
$\widetilde{B_0}$.
We can embed $\Pi$ naturally into this lattice:
$$E_1=h-R_2-R_3,  \quad
E_2=h-R_1-R_3,  \quad
E_3=h-R_1-R_2.$$
These reflect elliptic fibrations induced by pencils of planes in $\bP^3$
passing through two of the three nodes.
Note that
$$E_i=R_0+R_i, \quad i=1,2,3,$$
which means that each elliptic fibration admits a fiber of Kodaira type
$I_2$ containing $R_0$ as a component.

The connection between $B_0\subset \bP^3$ and $(2,2,2)$
K3 surfaces in $\bP^1 \times \bP^1 \times \bP^1$ goes further. 
There is a birational map
$$\bP^3 \dashrightarrow \bP^1 \times \bP^1 \times \bP^1,
$$
where the map onto each factors is given by the pencil of planes
through a pair of nodes of $B_0$. This maps $B_0$ birationally onto a 
$(2,2,2)$ nodal K3 surface $B_{\circ}$, as $R_0$ is in the fiber of each
of the elliptic fibrations. Conversely, generic nodal $(2,2,2)$ surfaces 
$B_{\circ} \subset \bP^1 \times \bP^1 \times \bP^1$
yield quartic surfaces with three nodes.
To summarize:
\begin{lemm}
Generic double solids $V_0\ra \bP^3$ branched over a quartic
surface with three nodes yield double covers $V_{\circ} \ra \bP^1 \times
\bP^1 \times \bP^1$ branched over a nodal $(2,2,2)$ surface,
and {\em vice versa}.
\end{lemm}

Voisin \cite[Th.~1.1]{Voisin}
has shown that a double solid
branched over a very general quartic surface with $r\le 7$
nodes fails to admit an integral decomposition of the diagonal.
Thus the same holds for a double cover of
$\bP^1 \times \bP^1 \times \bP^1$ branched over a very general
$(2,2,2)$ surface. We conclude that such threefolds fail
to be stably rational.

\bibliographystyle{alpha}
\bibliography{fano3}
\end{document}